\documentclass[a4paper, 12pt, oneside, reqno]{amsart}
\usepackage[foot]{amsaddr} 
\usepackage[top=35mm,bottom=30mm,left=25mm,right=25mm]{geometry}	

\usepackage{mathrsfs}
\usepackage{amsfonts}
\usepackage{amssymb}
\usepackage{amsxtra}
\usepackage{color}
\usepackage{dsfont}
\usepackage[colorinlistoftodos,  textsize=tiny]{todonotes} 					 
\presetkeys{todonotes}{color=blue!30}{}
 \usepackage{textcomp}
\usepackage[english,polish]{babel}
\usepackage[colorlinks,citecolor=blue,urlcolor=blue,bookmarks=true]{hyperref}
\hypersetup{
pdfpagemode=UseNone,
pdfstartview=FitH,
pdfdisplaydoctitle=true,
pdfborder={0 0 0}, 
pdftitle={Asymptotics of non-local perimeters},
pdfauthor={Wojciech Cygan, Tomasz Grzywny},
pdflang=en-US
}

\newcommand{\RR}{\mathbb{R}}

\newcommand{\Per}{\mathrm{Per}}

\newcommand{\ud}{\mathrm{d}}

\newcommand{\pl}[1]{\foreignlanguage{polish}{#1}}

\newtheorem{theorem}{Theorem}[section]
\newtheorem{proposition}[theorem]{Proposition}
\newtheorem{lemma}[theorem]{Lemma}
\newtheorem{corollary}[theorem]{Corollary}

\theoremstyle{definition}
\newtheorem{example}[theorem]{Example}
\newtheorem{remark}[theorem]{Remark}

\title[Asymptotics of non-local perimeters]{Asymptotics of non-local perimeters}

\author{Wojciech Cygan}
\thanks{Research supported by National Science Centre (Poland), grant no.\ 2019/33/B/ST1/02494}

\address{(Wojciech Cygan) 
		University of Wroc{\lll}aw,
		Faculty of Mathematics and Computer Science\\
		Institute of Mathematics,
		pl.\ Grunwaldzki 2/4, 50--384 Wroc{\lll}aw, Poland
\&		
		Technische Universit\"{a}t Dresden,
		Faculty of Mathematics\\
		Institute of Mathematical Stochastics,
		Zellescher Weg 25, 01069 Dresden, Germany}
\email{wojciech.cygan@uwr.edu.pl}

\author{Tomasz Grzywny}

\address{(Tomasz Grzywny)
Wroc{\lll}aw University of Science and Technology,
Faculty of Pure and Applied Mathematics\\
	Wyb. \pl{Wyspia\'{n}skiego} 27,
	50-370 \pl{Wroc\l{}aw}, Poland}
\email{tomasz.grzywny@pwr.edu.pl}


\numberwithin{equation}{section}
\allowdisplaybreaks

\begin{document}
\selectlanguage{english}

\begin{abstract}
We introduce a notion of non-local perimeter which is defined through an arbitrary positive Borel measure on $\RR^d$ which integrates the function $1\wedge |x|$. Such definition of non-local perimeter encompasses a wide range of perimeters which have been already studied in the literature, including fractional perimeters and anisotropic fractional perimeters. The main part of the article is devoted to the study of the asymptotic behaviour of non-local perimeters. As direct applications we recover well-known convergence results for  fractional perimeters and anisotropic fractional perimeters
\end{abstract}

\maketitle

\section{Introduction}

The concept of non-local perimeter of a given Borel set $E\subset \RR^d$ of finite Lebesgue measure which corresponds to the fractional Laplacian was proposed in \cite{Caffarelli_Savin}. This is the so-called $\alpha $-perimeter and it is defined via the following integral formula
\begin{equation}\label{Per_alpha}
\Per_\alpha (E) = \int_E\int_{E^c}\frac{1}{|x-y|^{d+\alpha}}\, \ud x\, \ud y,
\end{equation}
where $0<\alpha <1$ and $E^c$ is the complement of $E$. By $|x|$ we denote the Euclidean norm of $x\in \RR^d$. This object is strongly related to the fractional Sobolev norm and it has been intensively studied over last years, see \cite{Ambrosio-Per}, \cite{Bourgain_Brezis_Mironescu}, \cite{Caffarelli_Savin}, \cite{Caffarelli_Valdinocci},
 \cite{Figalli}, \cite{Frank}, \cite{Fusco-Per}, \cite{Ludwig_Per}, \cite{Visintin}, and \cite{Kreuml}, \cite{Capolli}, \cite{Carbotti}. We also refer to \cite{Schilling_1} and \cite{Schilling_2} for the case of fractional norms related to Feller generators.

The main motivation for the present article was an another interesting variant of non-local perimeter which is defined through  a given non-singular kernel. More precisely, if $J\colon \RR^d \to [0,\infty )$ is a radially symmetric and integrable function then the corresponding $J$-perimeter of a set $E$ is given by
\begin{equation}\label{Per_J}
\Per_J (E) = \int_E\int_{E^c} J(x-y)\, \ud y \, \ud x.
\end{equation}
In \cite{Rossi_paper_1} the authors established basic properties and convergence results for $J$-perimeters, see also \cite{Rossi_book}, \cite{Cesaroni},  \cite{Pagliari_paper}, \cite{Pagliari_2} and \cite{Pagliari_phD}.
For a treatment on a unified framework for non-local perimeters and curvatures we refer to \cite{Chambolle_etal}. 

Our principal goal is to introduce a notion of non-local  perimeter which is defined with the aid of a given positive Borel measure on $\mathbb{R}^d$. 
If we look at \eqref{Per_alpha} and \eqref{Per_J} from a probabilistic perspective then it is evident that the $\alpha$-perimeter is linked with an $\alpha$-stable L\'{e}vy process while $J$-perimeter is associated with a compound Poisson process. We thus aim to develop a unified approach that encompasses the both concepts as special cases. We emphasize that our methods are, however, purely analytical.

Before we define the object of our study we recall the definition of the classical perimeter.
The perimeter of a Borel set $E\subset \RR^d$ can be defined in the variational language as the total mass of the total variation measure of the indicator function $\mathbf{1}_E$.
More precisely,  the distributional gradient $Du$ of a function $u\in L^1(\RR^d) $ is a vector-valued Radon measure and its total variation is a positive measure defined as
\begin{align*}
\vert Du \vert = \sup \left\{ \int_{\RR^d}u(x)\, \mathrm{div}\, v(x)\, \ud x:\, v\in C_c^\infty (\RR^d, \RR^d),\, \Vert v\Vert_\infty \leq 1  \right\}.
\end{align*}
The perimeter of $E$ is then given by
\begin{align}\label{Perimeter_def}
\Per(E) = |D\mathbf{1}_E |(\RR^d).
\end{align}
It is known that for a set $E$ of finite Lebesgue measure, $\Per(E)$ is finite if, and only if, $\mathbf{1}_E\in \mathrm{BV}(\RR^d)$, where the 
space of functions of bounded variations is defined as 
\begin{align*}
\mathrm{BV} (\RR^d) = \{ u\in L^1(\RR^d):\, |Du|(\RR^d)<\infty \}.
\end{align*}
The space $\mathrm{BV}$ is endowed with the norm
$\Vert u \Vert_{\mathrm{BV}} = \Vert u\Vert_{L^1}+ \vert Du\vert $ and it holds $W^{1,1}(\RR^d) \subset \mathrm{BV}(\RR^d)$. Our main reference for functions of bounded variation is \cite{Ambrosio_2000}.

For any positive Borel measure $\nu$ on $\RR^d$ satisfying 
\begin{align}\label{ass_1}
\int (1\wedge |x|)\, \nu (\ud x)< \infty \quad \mathrm{and}\quad \nu (\{0\})=0
\end{align}
we consider 
the corresponding non-local $\nu$-perimeter of a Borel set $ E\subset \RR^d$ defined as
\begin{align*}
\Per_{\nu} (E) = \int_{E} \int_{E^c-x}\nu (\ud y)\, \ud x . 
\end{align*}
It was recently observed in \cite{Heat_Content_1} that 
such perimeters appear as limit objects in the asymptotics of the heat content related to L\'{e}vy processes of bounded variation. It was proved in \cite[Lemma 1]{Heat_Content_1} that 
for a set $E$ of finite Lebesgue measure and of finite perimeter, $\Per_\nu (E)$ is finite as well.
To the non-local $\nu$-perimeter we attach the space
\begin{align}
\mathrm{BV}_{\nu} (\RR^d) = \{ u\in L^1(\RR^d):\, \int_{\RR^d}\int_{\RR^d} |u(x+y)-u(y)|\, \nu (\ud x)\, \ud y <\infty \}.
\end{align}
It is equipped with the norm 
$\Vert u\Vert_{\mathrm{BV}_{\nu}} = \Vert u\Vert_{L^1} + \mathcal{F}_{\nu}(u)$, 
where 
\begin{align*}
\mathcal{F}_{\nu} (u) = \frac{1}{2}\int_{\RR^d}\int_{\RR^d} |u(x+y)-u(y)|\, \nu (\ud x)\, \ud y .
\end{align*}
In Section \ref{sec:Basic} we show that $\mathrm{BV}_\nu (\RR^d)$ is a Banach space and $\mathrm{BV}(\RR^d)\subset \mathrm{BV}_{\nu} (\RR^d)$. For sets of finite Lebesgue measure and of finite perimeter their $\nu$-perimeter can be computed through the formula $\Per_\nu (E)=\mathcal{F}_\nu (\mathbf{1}_E)$. Furthermore, we observe that for the $\nu$-perimeter a version of isoperimetric inequality holds in the case when the measure $\nu$ is given by a radially increasing kernel. We also find a co-area formula for the $\nu$-perimeter.

There has been a vivid interest in asymptotic behaviour and convergence results for fractional perimeters in recent years. 
The asymptotics as $\alpha \uparrow 1$ were found in \cite{Ponce} (see also \cite{Bourgain_Brezis_Mironescu} and \cite{Davila}; and  recent paper \cite{Cesaroni_second} for second order asymptotics). In this case the following convergence holds for any set $E$ of finite Lebesgue measure and of finite perimeter,
\begin{align}\label{Per_alpha_conv_1}
\lim_{\alpha \uparrow 1}\, (1-\alpha)\Per_\alpha (E) = \frac{K_{1,d}}{2}\, \Per (E),
\end{align}
where 
\begin{align}\label{K_1d}
K_{1,d}= \int_{\mathbb{S}^{d-1}}|e\cdot \theta|\sigma (\ud \theta),\quad |e|=1.
\end{align}
Here, $\sigma$
stands for the usual surface measure on the unit sphere. 
One can show (see e.g.\ \cite{Lombardini}) that $K_{1,d}= 2\varpi_{d-1}$, where  
$
\varpi_d = \pi^{d/2}/\Gamma \left(\frac{d}{2}+1\right)
$
is the Lebesgue measure of the unit ball in $\RR^d$. 
On the other hand, if $\alpha \downarrow 0$, the asymptotic result was found in \cite{Mazya} (see also \cite{Figalli} for a more detailed treatment) and it asserts that for any bounded set $E$ of finite perimeter, 
\begin{align}\label{Per_alpha_conv_0}
\lim_{\alpha \downarrow 0}\alpha \Per_\alpha (E) = \kappa_{d-1}|E|,
\end{align}
where $|E|$ stands for the Lebesgue measure of $E$ and $\kappa_{d-1}=\sigma (\mathbb{S}^{d-1}) =d\varpi_d$. For $J$-perimeters it was found in \cite{Rossi_paper_1} that under the assumption that $J$ has compact support and for bounded sets $E$ of finite perimeter the following convergence holds
\begin{equation}\label{Per_J_conv}
\lim_{\varepsilon\downarrow 0}\varepsilon^{-1}\Per_{J_\varepsilon}(E) = C_J^{-1}\Per (E),
\end{equation}
where $J_\varepsilon (x) = \varepsilon^{-d}J(x/\varepsilon)$ and $C_J = 2( \int_{\RR^d}J(x)|x_d|\, \ud x)^{-1}$ (here $x=(x_1,\ldots ,x_d)$).

In order to obtain asymptotics of non-local $\nu$-perimeters we first extend \cite[Theorem 2]{Ponce} to the current setting (see Theorem \ref{Thm:Ponce}) and with this result at hand we show that,  for any family of measures $\{\nu_\varepsilon\}_{\varepsilon >0}$ satisfying \eqref{ass_1} and 
such that the mass of normalized measures $(1\wedge |x|)\nu_\varepsilon (\ud x)$ concentrates at zero, there is a sequence $\varepsilon_j$ converging to zero such that
\begin{align}\label{intro_2}
\lim_{j\to \infty} C_{\varepsilon_j}^{-1}\Per_{\nu_{\varepsilon_j}}(E) = 
\frac{1}{2}
\int_{\mathbb{S}^{d-1}}\int_{\RR^d} |D \mathbf{1}_E \cdot \theta|\, \mu(\mathrm{d}\theta).
\end{align}
Here, $C_{\varepsilon_j}$ are normalizing constants and $\mu$ is a probability measure on the unit sphere which is constructed through the family $\{\nu_\varepsilon\}_{\varepsilon >0}$, see Section \ref{sec:Asymptotics} for details. Clearly, if the limit measure $\mu$ happens to be the (normalized) uniform measure on the unit sphere then the right-hand side of \eqref{intro_2} is equal to the right-hand side of \eqref{Per_alpha_conv_1} divided by $\kappa_{d-1}$.
It turns out that this approach applies to fractional perimeters and $J$-perimeters and we not only recover results in \eqref{Per_alpha_conv_1} and \eqref{Per_J_conv} but we also abandon the assumption that $J$ is compactly supported.
 
A fruitful observation in this context is the fact that the non-local perimeter of $E$ can be represented  with the aid of the so-called covariance function related to the set $E$, see \eqref{g_omega_defn}. This enables us to
investigate the case when the mass of the normalized measures $(1\wedge |x|)\nu_\varepsilon (\ud x)$ concentrates at infinity and as an application we recover \eqref{Per_alpha_conv_0}. We also find a corresponding result for $J$-perimeters. We show that under the assumption that the function $\ell (s)= \int_{|x|>s}J(x)\, \ud x$ is slowly varying at zero, for any set $E$ of finite Lebesgue measure and of finite perimeter it holds 
\begin{align*}
\lim_{\varepsilon \downarrow 0}\ell(\varepsilon)^{-1}\Per_{\widetilde{J}_\varepsilon }(E) = |E|,
\end{align*}
where $\widetilde{J}_\varepsilon (x) = \varepsilon^dJ(\varepsilon x)$. Recall that $\ell$ is slowly varying at zero if $\lim_{s\to 0}\ell (\lambda s)/\ell (s)=1$, for $\lambda >0$, see \cite{Bingham}.

The notion of non-local $\nu$-perimeter can be also successfully  exploited in the framework of anisotropic perimeters.
Anisotropic perimeter related to a given convex body is a natural generalization of the classical perimeter and it is defined via a norm whose unit ball is equal to a given convex body, see \cite{Figalli_Inventiones} and references therein.
Let $K\subset \RR^d$ be a convex compact set of non-empty interior (so-called convex body) and such that it is origin-symmetric. Let $\Vert \cdot \Vert_K$ denote the unique norm on $\RR^d$ with unit ball equal to $K$, that is
\begin{align*}
\Vert x\Vert_K = \inf \{ \lambda >0: \lambda^{-1}x\in K\},\quad x\in \RR^d.
\end{align*}
Let $K^*=\{y\in \RR^d:\sup_{x\in K} y\cdot x\leq 1\}$ be the polar body of $K$. Anisotropic perimeter of a Borel set $E\subset \RR^d$ with respect to $K$ is defined as
\begin{align*}
\Per (E,K) = \int_{\partial^* E}\Vert \mathsf{n}_{E}(x)\Vert_{K^*}\, \ud x.
\end{align*}
Here $\mathsf{n}_E(x)$ denotes the measure theoretic outer unit normal vector of $E$ at $x\in \partial^* E$, where $\partial^*E$ is the reduced boundary of $E$, see \cite[Section 3.5]{Ambrosio_2000}.
Similarly as the classical perimeter is linked with Sobolev norm, the anisotropic perimeter is related to anisotropic Sobolev (semi)norms which have been intensively studied, see \cite{Alvino}, \cite{Villani}, \cite{Figalli_Adv}, \cite[Appendix by M.\ Gromov]{Gromov}, \cite{Ludwig_Norm} and \cite{Ma}. 

There also exists a fractional counterpart of the anisotropic perimeter. For $0<\alpha <1$, the anisotropic $\alpha$-perimeter of $E$ with respect to $K$ is given by
\begin{align*}
\Per_{\alpha} (E,K) = \int_E \int_{E^c}\frac{1}{\Vert x-y\Vert ^{d+\alpha }_K}\,\ud x\, \ud y .
\end{align*}
In \cite{Ludwig_Per} it was proved that for any bounded set of finite perimeter the following results hold
\begin{align}\label{anisotropic_alpha_per_coverg}
\lim_{\alpha \uparrow 1}(1-\alpha) \Per_\alpha (E,K) = \Per (E,ZK)
\end{align}
and 
\begin{align}\label{Ludwig at zero}
\lim_{\alpha \downarrow 0}\alpha \Per_{\alpha}(E,K) = d|K||E|,
\end{align}
where $ZK$ is the so-called moment body of $K$, see \eqref{moment_body} for the definition.
Through the methods of the present paper we are able to recover convergence in \eqref{anisotropic_alpha_per_coverg} and \eqref{Ludwig at zero} and show that they are actually valid for all sets of finite measure and of finite perimeter.  Similarly, we establish the corresponding convergence of anisotropic Sobolev norms given in \cite[Theorem 8]{Ludwig_Per} and show that the assumption of compact support is superfluous.

\section{Basic properties of the non-local perimeter}\label{sec:Basic}
In this section we establish a few essential properties of the non-local perimeter such as isoperimetric inequality and co-area formula. We start with the following symmetry property. 
\begin{lemma}\label{Lemma1}
For any $E\subset \RR^d$ of finite Lebesgue measure we have
\begin{align*}
\Per_{\nu} (E) = \Per_{\nu} (E^c).
\end{align*}
\end{lemma}

\begin{proof}
Indeed, by Fubini's theorem we obtain
\begin{align*}
\Per_{\nu} (E)& = \int_{E} \int_{E^c-x}\nu (\ud y)\, \ud x
= \int \int \mathbf{1}_{E}(x) \mathbf{1}_{E^c-y}(x)\, \ud x\, \nu (\ud y)\\
 &=\int |E\cap (E-y)^c|\nu(dy). 
\end{align*}
Since $|E\cap (E-y)|=|E\cap (E+y)|$, we have
$$|E\cap (E-y)^c| = |E|- |E\cap(E-y)|=|E\cap(E+y)^c|=|E^c\cap(E-y)|.$$
Hence, $\Per_{\nu} (E) =\Per_\nu (E^c)$, as desired.
\end{proof}

\begin{remark}
(1) Lemma \ref{Lemma1}  evidently does not hold if $|E|=\infty$. Indeed, the equality $|E\cap (E-y)^c|  = |E\cap(E+y)^c|$ fails already  for $(0,\infty)$.\\
(2) Without loss of generality we could assume that the measure $\nu$ appearing in the definition of the $\nu$-perimeter is symmetric in the sense that $\nu (A) = \nu (-A)$, for any Borel set $A$. This follows from the fact that $\Per_\nu (A) = \Per_{\widetilde{\nu}}(A)$, where $\widetilde{\nu}(A) = \frac{1}{2}(\nu (A)+\nu (-A))$.
\end{remark}

The following result is an isoperimetric inequality for the non-local perimeter in the case when the measure $\nu$ is given by a radial kernel $j$. We emphasize that the function $j$ does not have to be integrable. Recall that for a set $A\subset \RR^d$ we denote  by $|A|$ its Lebesgue measure.  We omit the proof as it is the same as that of \cite[Proposition 3.1]{Cesaroni}, see also \cite[Theorem 2.4]{Rossi_paper_1}. 
\begin{proposition}[Isoperimetric inequality]
Let $\nu (\ud x) = j(x)\ud x$ where $j\geq 0$ is a radially non-increasing function. Then for any $E\subset\RR^d$ of finite Lebesgue measure
\begin{align*}
\Per_\nu (E)\geq \Per_\nu (B),
\end{align*}
where $B$ is an open ball centred at the origin such that $|B|= |E|$.
\end{proposition}

We further claim that under condition \eqref{ass_1} the space $\mathrm{BV}(\RR^d)$ is contained in $\mathrm{BV}_\nu(\RR^d)$. 
In particular, by Lemma \ref{Lemma1} we easily obtain that
$$\mathcal{F}_\nu (\mathbf{1}_E)=\mathrm{Per}_\nu (E),$$ 
for any $E$ of finite Lebesgue measure and of finite perimeter.
For any vector-valued measure $\Lambda$ we use notation $\int_{\RR^d} \Lambda = \Lambda (\RR^d)$.

\begin{proposition}
For any $u\in \mathrm{BV}(\RR^d)$ it holds
\begin{itemize}
\item[(i)] $\mathcal{F}_\nu (u)\leq C_\nu \Vert u\Vert_{\mathrm{BV}}$, where $C_\nu = \int (1\wedge |x|)\, \nu (\ud x)$.
\item[(ii)] $\mathcal{F}_\nu (u) \leq \tilde{C}_\nu \int_{\RR^d} |Du|$, where $\tilde{C}_\nu = \int |x|\nu (\ud x) \leq \infty$.
\end{itemize}
\end{proposition}

\begin{proof}
It is known \cite[Thm.~3.9]{Ambrosio_2000} that for any $u\in \mathrm{BV}(\RR^d)$ there exists a sequence $u_n\in C^\infty \cap W^{1,1}(\RR^d)$ such that 
\begin{align*}
u_n\to u\  \mathrm{in}\ L^1\quad \mathrm{and}\quad \lim_{n\to \infty}\int_{\RR^d} |\nabla u_n(x)|\, \ud x = \int_{\RR^d} |Du|\,. 
\end{align*}
We prove that for sequence $u_n$ it holds 
\begin{align*}
\lim_{n\to \infty} \mathcal{F}_\nu (u_n) = \mathcal{F}_\nu (u).
\end{align*}
Observe that
\begin{align}
\int_{\RR^d} |u_n(x+y)-u_n(y)|\, \ud y &\leq |x|\int_{\RR^d} \int_0^1 |\nabla u_n (y+tx)|\, \ud t\, \ud y\nonumber \\
&= |x|\int_0^1 \int_{\RR^d}  |\nabla u_n (w)|\, \ud w\, \ud t \leq C|x|\, \vert Du\vert.\label{eq1}
\end{align}
Further, we have
\begin{align}\label{eq2}
\int_{\RR^d} |u_n(x+y)-u_n(y)|\, \ud y &\leq C \Vert u \Vert_{L^1}.
\end{align}
Thus, we may apply the dominated convergence theorem to arrive at
\begin{align*}
\lim_{n\to \infty} \mathcal{F}_\nu (u_n) &= \frac{1}{2}\int_{\RR^d} \lim_{n\to \infty}\int_{\RR^d} |u_n(x+y)-u_n(y)|\, \ud y\, \nu (\ud x) \\
&= \frac{1}{2}\int_{\RR^d}\int_{\RR^d} |u(x+y)-u(y)|\, \nu (\ud x)\, dy = \mathcal{F}_\nu (u).
\end{align*}
Hence, it suffices to operate on $u\in C^\infty \cap W^{1,1}(\RR^d)$. Then inequality (i) follows easily by \eqref{eq1} and \eqref{eq2}. Similarly, (ii) is a consequence of \eqref{eq1}.
\end{proof}

\begin{corollary}
Let $E\subset\RR^d$ be such that $|E|<\infty$ and $\Per (E)<\infty$. Then
\begin{align*}
\Per_{\nu}(E)\leq C_\nu (\Per (E) + |E|).
\end{align*}
\end{corollary}

We next show for completeness that $\mathrm{BV}_\nu (\RR^d)$ is a Banach space.

\begin{lemma}\label{lemma:Banach}
The space $\mathrm{BV}_\nu (\RR^d)$ equipped with the norm $\Vert u\Vert_{\mathrm{BV}_{\nu}} = \Vert u\Vert_{L^1} + \mathcal{F}_{\nu}(u)$ is a Banach space.
\end{lemma}

\begin{proof}
The function $u\mapsto \Vert u\Vert_{\mathrm{BV}_\nu}$ is evidently a norm and thus we only need to show completeness. Let $\{u_n\}$ be a Cauchy sequence in $\mathrm{BV}_\nu (\RR^d)$ and let $u$ be its $L^1$ limit. By Fatou's lemma we have
\begin{align*}
2\mathcal{F}_\nu (u) 
&=
\int_{\RR^d}  \lim_{n\to \infty} \int_{\RR^d} | u_n(x+y)-u_n(x)|\, \ud x\, \nu (\ud y)\\
&\leq
\liminf_{n\to \infty}\int_{\RR^d}\int_{\RR^d} | u_n(x+y)-u_n(x)|\, \nu (\ud y)\, \ud x = 2\liminf_{n\to \infty}\mathcal{F}_{\nu}(u_n)
\end{align*}
and as the sequence $\{u_n\}$ is bounded in $\mathrm{BV}_\nu (\RR^d)$ we infer the result.
\end{proof}

\begin{proposition}[Co-area formula]
For $u\in L^1$ we set $S_t(u)=\{ x\in \RR^d:\, u(x)>t \}$. Then
\begin{align*}
\mathcal{F}_{\nu} (u) = \int_{\RR} \Per_{\nu} (S_t(u)) \, \ud t.
\end{align*}
\end{proposition}

\begin{proof}
Since
\begin{align*}
u(x) = \int_0^\infty \mathbf{1}_{S_t(u)}(x)\,\ud t -\int_{-\infty}^0 (1-\mathbf{1}_{S_t(u)}(x))\,\ud t ,
\end{align*}
we have
\begin{align*}
|u(x)-u(y)| = \int_{-\infty }^\infty |\mathbf{1}_{S_t(u)}(x) -\mathbf{1}_{S_t(u)}(y)|\,\ud t .
\end{align*}
Thus, by Tonelli's theorem, 
\begin{align*}
\mathcal{F}_{\nu} (u) = \frac{1}{2}\int_\RR \int_{\RR^d} \int_{\RR^d} |\mathbf{1}_{S_t(u)}(x+y) -\mathbf{1}_{S_t(u)}(y)  |\, \nu (\ud x)\, \ud y\, \ud t.
\end{align*}
In view of Lemma \ref{Lemma1} we obtain
\begin{align*}
\int_{\RR^d} \int_{\RR^d} |\mathbf{1}_{S_t(u)}(x+y) &-\mathbf{1}_{S_t(u)}(y)  |\, \nu (\ud x)\, \ud y \\
&= \int_{S_t(u)}\int_{S_t(u)^c-y}\nu (\ud x)\, \ud y 
+ \int_{S_t(u)^c}\int_{S_t(u)-y}\nu (\ud x)\, \ud y\\
&= \Per_{\nu} (S_t(u)) + \Per_{\nu} (S_t(u)^c) = 2\Per_{\nu} (S_t(u)) 
\end{align*}
and the result follows.
\end{proof}

\section{Asymptotics of the non-local perimeter}\label{sec:Asymptotics}
In this section we show that the non-local perimeter which we introduce in the present article converge towards the classical perimeter (or to the Lebesgue measure) if we use an appropriate scaling procedure. We then apply our results to establish  convergence of fractional perimeters, $J$-perimeters and anisotropic fractional perimeters.

\subsection{Convergence towards the classical perimeter}
We first formulate an approximation result in the space of functions of bounded variation which can be seen as a generalization of the result by Ponce \cite{Ponce}, see also \cite{Davila} and \cite{Bourgain_Brezis_Mironescu}. By $B_R$ we denote the closed ball centred at zero and of radius $R>0$.

\begin{theorem}\label{Thm:Ponce}
Let $\{\lambda\}_{\varepsilon >0}$ be a family of probability measures on $\RR^d$ such that $\lambda_\varepsilon (\{0\})=0$. Suppose that
\begin{align}\label{tight_at_zero}
\lim_{\varepsilon \downarrow  0}\lambda_\varepsilon (B_R^c)=0,\quad \mathrm{for\ any}\ R>0.
\end{align}
Further, let $\mu_\varepsilon$ be a probability measure on the unit sphere given by
\begin{align}\label{mu_projection on the sphere}
\mu_{\varepsilon}(E)= \lambda_{\varepsilon}((0,\infty) E),\quad E\subset \mathbb{S}^{d-1},
\end{align}
 where $(0,\infty)E = \{re:\,  r>0\ \text{and}\ e\in E\}$ is a cone determined by $E$. 
 Then there exists a sequence $\varepsilon_j$ converging to zero such that for any $f\in \mathrm{BV}(\mathbb{R}^d)$
\begin{align}\label{conv_Ponce_subseq}
\lim_{j\to \infty} \int_{\mathrm{R}^d} \int_{\mathrm{R}^d} \frac{\left \vert f(x+y)-f(x) \right\vert}{|y|}\lambda_{\varepsilon_j}(\mathrm{d}y)\, \mathrm{d}x 
=
\int_{\mathbb{S}^{d-1}}\int_{\RR^d}|Df\cdot \theta|\, \mu(\mathrm{d}\theta),
\end{align}
where $\mu $ is a probability measure  on the unit sphere which is equal to the weak limit of the sequence $\{\mu_{\varepsilon_j}\}$.
\end{theorem}
The proof is postponed to the Appendix given in Section \ref{sec:Appendix}.
\medskip

\begin{remark}
In many applications we are able to conclude the weak convergence of the whole sequence $\{\mu_\varepsilon \}$ towards the limit measure  $\mu$. In such cases convergence in \eqref{conv_Ponce_subseq} holds for $\varepsilon \downarrow 0$ and not only along a subsequence. This fact follows directly from the proof of Theorem \ref{Thm:Ponce}.
\end{remark}

We apply Theorem \ref{Thm:Ponce} in the case when the probability measure $\lambda_\varepsilon$ is absolutely continuous with respect to a L\'{e}vy measure. We start by giving a general result and then present a few examples.

\begin{theorem}\label{Thm:focus_zero}
Let $\{\nu_\varepsilon\}_{\varepsilon >0}$ be a family of L\'{e}vy measures satisfying \eqref{ass_1} and let 
\begin{align*}
\lambda_\varepsilon (\mathrm{d}x) = C_\varepsilon^{-1}\left( R_\varepsilon \wedge |x|\right)\nu_\varepsilon (\mathrm{d}x),
\end{align*}
where $C_\varepsilon = \int_{\mathbb{R}^d}\left( R_\varepsilon \wedge |x|\right)\nu_\varepsilon (\mathrm{d}x)$ and $R_\varepsilon \in [1,\infty]$. 
Let $\mu_{\varepsilon}$ be the corresponding probability measure on the unit sphere defined in \eqref{mu_projection on the sphere}.
Under condition \eqref{tight_at_zero}, there exists a sequence $\varepsilon_j$ converging to zero such that 
for any $f\in \mathrm{BV}(\RR^d)$ we have
\begin{align}\label{Result_at_zero}
\lim_{j\to \infty} C_{\varepsilon_j}^{-1}\mathcal{F}_{\nu_{\varepsilon_j}} (f)
= 
\frac{1}{2}\int_{\mathbb{S}^{d-1}}\int_{\RR^d}|Df\cdot \theta|\, \mu(\mathrm{d}\theta),
\end{align}
where $\mu$ is the weak limit of the sequence $\{\mu_{\varepsilon_j}\}$.
\end{theorem}

\begin{proof}
We split the integral as follows
\begin{align}\label{split_1}
C_{\varepsilon_j}^{-1}\mathcal{F}_{\nu_{\varepsilon_j}} (f)
&=
\frac{1}{2}\left( \int_{\RR^d}\int_{B_{R_{\varepsilon_j}}} +  \int_{\RR^d}\int_{B_{R_{\varepsilon_j}}^c}\right)
\frac{\vert f(x+y)-f(x)|}{R_{\varepsilon_j}\wedge |y|}\lambda_{\varepsilon_j}(\mathrm{d}y)\mathrm{d}x.
\end{align}
We have
\begin{align*}
 \int_{\RR^d}\int_{B_{R_{\varepsilon_j}}} \!\!\!\!\!\!
 \frac{\vert f(x+y)-f(x)|}{|y|}\lambda_{\varepsilon_j}(\mathrm{d}y)\mathrm{d}x
 =\!\!
 \left(  \int_{\RR^d}\int_{\RR^d}\!\! - \!\!\int_{\RR^d}\int_{B_{R_{\varepsilon_j}}^c}  
 \right)
 \frac{\vert f(x+y)-f(x)|}{|y|}\lambda_{\varepsilon_j}(\mathrm{d}y)\mathrm{d}x
\end{align*}
and the first integral converges by Theorem \ref{Thm:Ponce}  to (two times) the right hand side of \eqref{Result_at_zero} while the second integral converges to zero as
\begin{align*}
\int_{\RR^d}\int_{B_{R_{\varepsilon_j}}^c}  
\frac{\vert f(x+y)-f(x)|}{|y|}\lambda_{\varepsilon_j}(\mathrm{d}y)\mathrm{d}x
\leq
2\Vert f\Vert_{L^1}\frac{\lambda_{\varepsilon_j}(B_{R_{\varepsilon_j}}^c)}{R_{\varepsilon_j}}\leq 2\Vert f\Vert_{L^1} \lambda_{\varepsilon_j}(B_1^c)\to 0.
\end{align*}
The second integral in \eqref{split_1} is negligible by the argument from the previous line.
\end{proof}

\begin{corollary}\label{cor:per}
In the notation of Theorem \ref{Thm:focus_zero}, let $f=\mathbf{1}_E$ 
where $E\subset \RR^d$ is a set of finite Lebesgue measure and of  finite perimeter (i.e.\ $\mathbf{1}_E\in \mathrm{BV}(\RR^d)$). Then
\begin{align*}
\lim_{j\to \infty} C_{\varepsilon_j}^{-1}\mathrm{Per}_{\nu_{\varepsilon_j}} (E)
= 
\frac{1}{2}
\int_{\mathbb{S}^{d-1}}\int_{\RR^d} |D \mathbf{1}_E \cdot \theta|\, \mu(\mathrm{d}\theta).
\end{align*}
In particular, if $\mu$ is the (normalized) uniform measure  on $\mathbb{S}^{d-1}$ then 
\begin{align*}
\lim_{j\to \infty} C_{\varepsilon_j}^{-1}\mathrm{Per}_{\nu_{\varepsilon_j}} (E)
= 
\frac{K_{1,d}}{2\kappa_{d-1}}\Per (E),
\end{align*}
where $K_{1,d}$ is the constant from \eqref{K_1d}.
\end{corollary}
We next illustrate Theorem \ref{Thm:focus_zero} and Corollary \ref{cor:per} by a few examples.
We start with the following result which is an application of Theorem \ref{Thm:focus_zero} to stable L\'{e}vy measures.

\begin{proposition}\label{Prop:stable_at_zero}
Let $\nu_\alpha$ be an $\alpha$-stable L\'{e}vy measure with its spectral decomposition given by
\begin{align}\label{alpha_stable_Levy_meas_case_at zero}
\nu_\alpha (A) = \int_{\mathbb{S}^{d-1}}\int_{0}^\infty \mathbf{1}_A(r \theta) r^{-1-\alpha}\mathrm{d}r \, \eta (\mathrm{d}\theta),\quad \alpha \in (0,1),
\end{align}
where $\eta $ is a probability measure on $\mathbb{S}^{d-1}$. Then, for any $f\in \mathrm{BV}(\RR^d)$,
\begin{align*}
\lim_{\alpha \uparrow 1}(1-\alpha )\mathcal{F}_{\nu_\alpha}(f)=\frac{1}{2}\int_{\mathbb{S}^{d-1}}\int_{\RR^d} |Df \cdot \theta|\, \mu (\mathrm{d}\theta),
\end{align*}
where $\mu$ is the measure on $\mathbb{S}^{d-1}$ constructed according to Theorem \ref{Thm:focus_zero}. In particular, if the measure $\eta$ is uniform on the unit sphere then,
for any set $E\subset \RR^d$ of finite Lebesgue measure and  such that $\Per (E)<\infty$, it holds
\begin{align*}
\lim_{\alpha \uparrow 1}\, (1-\alpha )\Per_{\nu_{\alpha}}(E) = \frac{K_{1,d}}{2\kappa_{d-1}}\Per (E).
\end{align*}
\end{proposition}

\begin{proof}
We aim to apply Theorem \ref{Thm:focus_zero}. We define measures $\lambda_\alpha$ as follows
\begin{align*}
\lambda_\alpha (A) = C_{\alpha}^{-1}\int_{\mathbb{S}^{d-1}}\int_0^\infty (1\wedge r)\mathbf{1}_{A}(r\theta) r^{-1-\alpha}\ud r\, \eta (\theta),\quad C_\alpha = \int_0^\infty (1\wedge r)r^{-1-\alpha }\ud r. 
\end{align*}
We first observe that
\begin{align*}
C_\alpha &= \int_1^\infty r^{-1-\alpha }\ud r + \int_0^1r^{-\alpha}\ud r = \frac{1}{\alpha}+\frac{1}{1-\alpha}\sim \frac{1}{1-\alpha},\quad \alpha \uparrow 1.
\end{align*}
Further, for any $R>0$,
\begin{align*}
\lambda_\alpha (B_R^c) = C_\alpha^{-1}\cdot
\begin{cases}
\int_R^\infty r^{-1-\alpha} \ud r = \frac{1}{\alpha}R^{-\alpha},\quad R>1;\\
\int_R^1 r^{-\alpha}\ud r + \int_1^\infty r^{-1-\alpha}\ud r= \frac{1}{1-\alpha}(1-R^{-\alpha +1})+ \frac{1}{\alpha},\quad R\leq 1.
\end{cases}
\end{align*}
We easily infer that 
\begin{align*}
\lim_{\alpha\uparrow 1}\lambda_{\alpha}(B_R^c)=0,\quad R>0.
\end{align*}
The measures $\mu_\alpha$ are given by
\begin{align*}
\mu_\alpha (E) =  \eta(E),\quad E\subset \mathbb{S}^{d-1},
\end{align*}
which evidently implies $\mu_\alpha \xrightarrow[\alpha \uparrow 1]{w}\eta$ 
and the result follows. 
\end{proof}

\begin{example}[Asymptotics of $\alpha$-perimeters for $\alpha \uparrow 1$]
As a direct application of Proposition \ref{Prop:stable_at_zero} we obtain the well-known convergence of $\alpha$-perimeters to the classical perimeter for sets of finite perimeter, see \cite{Caffarelli-Valdinoci}, \cite{Ambrosio-Per}, \cite{Valdinoci_Milan}, \cite{Davila}, \cite{Bourgain_Brezis_Mironescu}.
Let the measure $\nu_\alpha$ be rotationally invariant and given by
\begin{align}\label{rot_inv_stable_lev_meas}
\nu_\alpha (\mathrm{d}x) = \frac{\alpha\, \mathrm{d}x}{\kappa_{d-1}|x|^{d+\alpha}},\quad \alpha \in (0,1),
\end{align}
where $\kappa_{d-1}=2\pi^{d/2}/\Gamma (d/2)$ is the surface area of $\mathbb{S}^{d-1}$. For such measure it holds
\begin{align}\label{Our_Per_to_Alpha}
\Per_{\nu_\alpha}(E) = \frac{\alpha}{\kappa_{d-1}}\Per_\alpha (E),
\end{align}
where $\nu_\alpha$ is given by \eqref{alpha_stable_Levy_meas_case_at zero} with $\eta$ equal to ($\alpha$ times) the normalized surface measure on the unit sphere. Then, by Proposition \ref{Prop:stable_at_zero}, we recover \eqref{Per_alpha_conv_1} for any set $E\subset \RR^d$ of finite Lebesgue measure and such that $\Per(E)<\infty$.
\end{example}

\begin{example}[Asymptotics of $J$-perimeters]\label{J_perim}
Let $\nu (\ud x) = J(x)\, \ud x$, where $J$ is a positive function such that 
$C_J=\int_{\RR^d}|x|J(x)\, \ud x <\infty$, see \cite{Cesaroni}. For any $\varepsilon >0$ let $J_\varepsilon (x) = \varepsilon^{-d}J(x/\varepsilon)$ and $\nu_\varepsilon (\ud x) = J_\varepsilon (x)\, \ud x$. Clearly, $\Per_{\nu_\varepsilon}(E) = \Per_{J_\varepsilon}(E)$. In this case we can apply Theorem \ref{Thm:focus_zero} with $R_\varepsilon =\infty$ and we easily verify that the corresponding measures $\lambda_\varepsilon$ satisfy condition \eqref{tight_at_zero}. Further, for any $\varepsilon >0$ we have
\begin{align*}
\mu_\varepsilon (E) = C_J^{-1}\int_{(0,\infty)E}|x|J(x)\, \ud x =: \mu_J (E),\quad E\subset \mathbb{S}^{d-1}.
\end{align*}
Hence, by Corollary \ref{cor:per}, for any $E\subset \RR^d$ of finite Lebesgue measure and of finite perimeter, 
\begin{align*}
\lim_{\varepsilon \downarrow 0}\varepsilon ^{-1}\Per_{\nu_\varepsilon}(E) = 
\frac{C_J}{2}
\int_{\mathbb{S}^{d-1}}\int_{\RR^d} |D \mathbf{1}_E \cdot \theta|\, \mu_J(\mathrm{d}\theta).
\end{align*} 
In particular, if the kernel $J$ is radially symmetric then this leads to \eqref{Per_J_conv} and we observe  that we do not need to assume that $J$ is compactly supported.

We can use the scaling procedure of Example \ref{J_perim} in a slightly more general context. 
Let $\nu$ be a measure satisfying \eqref{ass_1} and let $\nu_\varepsilon (A)=\nu (A/\varepsilon)$ for any $\varepsilon >0$.
If $\int |x|\nu (\mathrm{d}x)<\infty$ then we can apply Theorem \ref{Thm:focus_zero} with $R_\varepsilon =\infty$ and we obtain $C_\varepsilon =\varepsilon \int |x|\nu (\mathrm{d}x)$. The corresponding convergence in \eqref{Result_at_zero} follows.
\end{example}

\begin{example}\label{example:rescale}

Let $\nu$ be a measure satisfying \eqref{ass_1} and let $\nu_\varepsilon (A)=\nu (A/\varepsilon)$ for any $\varepsilon >0$.
We assume that
\begin{align*}
\nu (A) =  \int_{\mathbb{S}^d-1} \int_0^\infty\mathbf{1}_A(r\theta)\varrho(\mathrm{d}r)\eta(\mathrm{d}\theta),
\end{align*}
where $\varrho$ is a positive measure on $(0,\infty)$ and 
$\eta$ is a probability measure on $\mathbb{S}^{d-1}$. 
Then 
\begin{align*}
\lambda_\varepsilon (B) = C_\varepsilon^{-1} 
 \int_{\mathbb{S}^{d-1}}
\int_0^\infty 
 (R_\varepsilon \wedge r)\mathbf{1}_{\varepsilon^{-1}B}(r\theta)
\varrho_\varepsilon(\mathrm{d}r)
\eta(\mathrm{d}\theta),
\end{align*}
where $C_\varepsilon = \int_0^\infty (R_\varepsilon \wedge r)\varrho_{\varepsilon}(\mathrm{d}r)$ and $\varrho_\varepsilon (\mathrm{d}r) = \varrho (\varepsilon \mathrm{d}r)$.
We observe that in this case the measures $\mu_\varepsilon$ of Theorem \ref{Thm:focus_zero} are all equal to $\eta$. In particular, if $\eta$ is rotationally invariant then \eqref{Result_at_zero} becomes
\begin{align*}
\lim_{\varepsilon \downarrow 0} C_{\varepsilon}^{-1}\mathcal{F}_{\nu_{\varepsilon}} (f)
= 
\frac{K_{1,d}}{2\kappa_{d-1}}
\int_{\mathbb{R}^d}|Du|.
\end{align*}
This results in the following approximation of the classical perimeter: for any set $E\subset \RR^d$ of finite Lebesgue measure and such that $\Per (E)<\infty$ it holds 
\begin{align*}
\lim_{\varepsilon \downarrow 0} C_{\varepsilon}^{-1}\mathrm{Per}_{\nu_{\varepsilon}} (E)
= 
\frac{K_{1,d}}{2\kappa_{d-1}}
\Per(E).
\end{align*}
\end{example}

\subsection*{Anisotropic fractional perimeters}
We can utilize Theorem \ref{Thm:focus_zero} also in a more general context of anisotropic fractional perimeters. 
We start by recalling an another (equivalent) definition of the classical perimeter. For a set $E$ such that $\mathbf{1}_E\in \mathrm{BV}(\RR^d)$ one can define its perimeter as
\begin{align}\label{df:Per_outer}
\Per (E)=\int_{\partial ^*E}|\mathsf{n}_E(x)|\mathcal{H}^{d-1}(\ud x),
\end{align}
where $\partial^*E$ is the reduced boundary of $E$, $\mathsf{n}_E(x)$ denotes the measure theoretic outer unit normal vector of $E$ at $x\in \partial^* E$ and $\mathcal{H}^{d-1}$ is the $(d-1)$-dimensional Hausdorff measure. For a detailed treatment on the fine structure of the classical perimeter we refer to \cite[Section 3.5]{Ambrosio_2000} and \cite{Figalli_Inventiones}.

%
Let $K\subset \RR^d$ be a convex compact set of non-empty interior (so-called convex body) and such that it is origin-symmetric.
Anisotropic perimeter of a Borel set $E\subset \RR^d$ with respect to $K$ is defined as
\begin{align*}
\Per (E,K) = \int_{\partial^* E}\Vert \mathsf{n}_{E}(x)\Vert_{K^*}\, \ud x.
\end{align*}
Let $ZK$ be the so-called moment body of $K$. It is defined as the unique convex body satisfying 
\begin{align}\label{moment_body}
\Vert y \Vert_{Z^*K} = \frac{d+1}{2}\int_K |y\cdot x|\, \ud x,\quad  y\ \in \RR^d,
\end{align}
where $Z^*K$ is the polar body of $ZK$. Our aim is to show (through the methods of the present paper) that convergence in \eqref{anisotropic_alpha_per_coverg} holds actually for all sets of finite perimeter.  

\begin{proposition}
For any Borel set $E\subset \RR^d$ of finite Lebesgue measure and of finite perimeter it holds
\begin{align*}
\lim_{\alpha \uparrow 1}(1-\alpha) \Per_\alpha (E,K) = \Per (E,ZK).
\end{align*}
\end{proposition}
\begin{proof}
Let $\nu_{\alpha}(A,K)$ be 
a measure (with respect to the convex body $K$) given by
\begin{align*}
\nu_{\alpha}(A,K) =\int_{\mathbb{S}^{d-1}}\int_0^\infty \mathbf{1}_A(r\theta)r^{-\alpha -1}\ud r\, \frac{\ud \theta}{\Vert \theta \Vert_K^{d+\alpha}}.
\end{align*}
We easily verify that
\begin{align}\label{Our Per is MLudiwig}
\Per_\alpha (E,K) = \Per_{\nu_\alpha (\cdot ,K)}(E).
\end{align} 
Clearly, $\nu_\alpha (\cdot ,K)$ is (up to a normalising constant) a special case of \eqref{alpha_stable_Levy_meas_case_at zero} and thus we are in the scope of Theorem \ref{Thm:focus_zero}. We easily find that the measure $\lambda _\alpha(\cdot ,K)$ appearing in Theorem \ref{Thm:focus_zero} satisfies
\begin{align*}
\lambda _\alpha(B_R^c ,K) = \frac{\int_R^\infty (1\wedge r)r^{-\alpha -1}\ud r}{\int_0^\infty (1\wedge r)r^{-\alpha -1}\ud r}
\end{align*}
and we show similarly as in the proof of Proposition \ref{Prop:stable_at_zero} that $\lambda _\alpha(B_R^c ,K)$
converges to $0$. The corresponding measure $\mu_\alpha (\cdot ,K)$ is given by
\begin{align*}
\mu_\alpha (S,K) =C_\alpha (K)^{-1}\int_{S}\frac{\ud \theta}{\Vert \theta \Vert_K^{d+\alpha}},\qquad S\subset \mathbb{S}^{d-1},
\end{align*}
where $C_\alpha (K) = \int_{\mathbb{S}^{d-1}}\frac{\ud \theta}{\Vert \theta \Vert_K^{d+\alpha}}$.
Since $\lim_{\alpha \uparrow 1}\Vert \theta \Vert_K^{-d-\alpha} = \Vert \theta \Vert_K^{-d-1}$ for every $\theta \in \mathbb{S}^{d-1}$, we infer that
\begin{align*}
\mu_\alpha (\cdot ,K)\xrightarrow[\alpha \uparrow 1]{w}\mu (\cdot ,K),
\end{align*}
where
$\mu (S,K)= C(K)^{-1}\int_{S}\Vert \theta \Vert^{-d-1}_K\ud \theta$ and $C(K)= \int_{\mathbb{S}^{d-1}}\Vert \theta \Vert^{-d-1}_K\ud \theta$. We obtain that for any $f\in \mathrm{BV}(\RR^d)$,
\begin{align}\label{F_nu_conv_Ludwig}
\lim_{\alpha \uparrow 1}C_\alpha^{-1}C_\alpha (K)^{-1}\mathcal{F}_{\nu_{\alpha}(\cdot ,K)}(f)= \frac{1}{2}\int_{\mathbb{S}^{d-1}}\int_{\RR^d} |Df\cdot \theta|\mu (\ud \theta ,K),
\end{align}
where $C_\alpha =\int_0^\infty (1\wedge r)r^{-1-\alpha }\ud r$. Hence, by taking $f=\mathbf{1}_E\in \mathrm{BV}(\RR^d)$, we arrive at
\begin{align*}
\lim_{\alpha \uparrow 1}\, (1-\alpha)\, \Per _{\alpha}(E,K) = 
\frac{1}{2}
\int_{\mathbb{S}^{d-1}}\int_{\RR^d} |D\mathbf{1}_E\cdot \theta|\, \frac{\ud \theta}{\Vert \theta \Vert _K^{d+1}}.
\end{align*}
Further, by employing \eqref{df:Per_outer} together with Fubini's theorem, we obtain
\begin{align*}
\frac{1}{2}\int_{\mathbb{S}^{d-1}}\int_{\RR^d} |D\mathbf{1}_E\cdot \theta|\, \frac{\ud \theta}{\Vert \theta \Vert _K^{d+1}}
=
\frac{1}{2}
\int_{\partial^*E}
\int_{\mathbb{S}^{d-1}} |\mathsf{n}_E(x)\cdot \theta| \frac{\ud \theta}{\Vert \theta \Vert _K^{d+1}}\, \mathcal{H}^{d-1}(\ud x).
\end{align*}
Finally, we have
\begin{align}\label{polar_outer_normal}
\int_K |\mathsf{n}_E(x)\cdot y|\ud y =\int_{\mathbb{S}^{d-1}}\int_0^{\frac{1}{\Vert \theta \Vert_K}}|\mathsf{n}_E(x)\cdot \theta|r^{d}\ud r\, \ud \theta =\frac{1}{d+1}\int_{\mathbb{S}^{d-1}} |\mathsf{n}_E(x)\cdot \theta| \frac{\ud \theta}{\Vert \theta \Vert_K^{d+1}}
\end{align}
which implies
\begin{align*}
\frac{1}{2}\int_{\mathbb{S}^{d-1}}\int_{\RR^d} |D\mathbf{1}_E\cdot \theta|\, \frac{\ud \theta}{\Vert \theta \Vert _K^{d+1}}
=
\int_{\partial^*E}\Vert\mathsf{n}_E(x)\Vert_{Z^*K}\mathcal{H}^{d-1}(\ud x) = \Per (E,ZK)
\end{align*}
and the result follows.
\end{proof}

\subsection*{Anisotropic Sobolev norms}
We finally show that similar methods apply in the context of anisotropic Sobolev norms. We aim to prove a stronger version of \cite[Theorem 8]{Ludwig_Per} as we abandon the assumption of compact support. We recall that for any $f\in \mathrm{BV}(\RR^d)$ we define its anisotropic Sobolev semi-norm (with respect to a given convex body $K$ which is origin-symmetric) as
\begin{align*}
\Vert f\Vert_{\mathrm{BV}, ZK}
=
\int_{\RR^d}\left\Vert \frac{Df}{|Df|}\right\Vert_{Z^*K}\! \ud |Df|.
\end{align*}
Here the vector  $Df/|Df|$ is the Radon-Nikodym derivative of the $\RR^d$-valued vector measure $Df$ with respect to the positive measure $|Df|$.

\begin{proposition}
For any $f\in \mathrm{BV}(\RR^d)$ we have
\begin{align}
\lim_{\alpha \uparrow 1}\, (1-\alpha) \int_{\RR^d}\int_{\RR^d}\frac{|f(x+y)-f(x)|}{\Vert y\Vert_K^{d+\alpha}}\, \ud y\, \ud x
=
2
\Vert f\Vert_{\mathrm{BV}, ZK}.
\end{align} 
\end{proposition}

\begin{proof}
We observe that by \eqref{F_nu_conv_Ludwig}
\begin{align*}
\lim_{\alpha \uparrow 1}\, (1-\alpha) \int_{\RR^d}\int_{\RR^d}\frac{|f(x+y)-f(x)|}{\Vert y\Vert_K^{d+\alpha}}\, \ud y\, \ud x
= 
\int_{\mathbb{S}^{d-1}}\int_{\RR^d} |Df \cdot \theta|\, \frac{\ud \theta}{\Vert \theta \Vert _K^{d+1}}.
\end{align*}
We are thus left to identify the limit. We first note that, in view of \cite[Proposition 1.23]{Ambrosio_2000}, 
\begin{align}\label{Radon-Nikodym}
\ud |Df\cdot \theta |  = \left \vert \frac{Df}{|Df|}\cdot \theta\right\vert \ud |Df|.
\end{align}
This implies
\begin{align*}
\int_{\mathbb{S}^{d-1}}\int_{\RR^d} |Df \cdot \theta|\, \frac{\ud \theta}{\Vert \theta \Vert _K^{d+1}}
&=
\int_{\mathbb{S}^{d-1}}\int_{\RR^d} \left\vert \frac{Df}{|Df|} \cdot \theta\right\vert \, \ud |Df|\, \frac{\ud \theta}{\Vert \theta \Vert _K^{d+1}}\\
&=
\int_{\RR^d} \int_{\mathbb{S}^{d-1}} \left\vert \frac{Df}{|Df|} \cdot \theta\right\vert \, \frac{\ud \theta}{\Vert \theta \Vert _K^{d+1}}\, \ud |Df|.
\end{align*}
We then proceed similarly as in \eqref{polar_outer_normal} to obtain
\begin{align*}
\int_{\mathbb{S}^{d-1}} \left\vert \frac{Df}{|Df|} \cdot \theta\right\vert \, \frac{\ud \theta}{\Vert \theta \Vert _K^{d+1}} = (d+1)\int_K \left\vert \frac{Df}{|Df|} \cdot y\right\vert\, \ud y = 2 \left\Vert \frac{Df}{|Df|}\right\Vert_{Z^*K} 
\end{align*}
and the result follows.
\end{proof}

\subsection{Convergence towards the Lebesgue measure}

In this paragraph we focus on convergence of non-local perimeters towards the Lebesgue measure in the case when the mass of the underlying L\'{e}vy measure concentrates at infinity. In the rest of the paper we make use of a covariance function of a set and thus we briefly recall its definition and basic properties.

For any $E \subset\RR^d$ of finite Lebesgue measure its covariance function $g_E$ is given by
\begin{align}\label{g_omega_defn}
g_E (y)=|E \cap (E + y)|=\int_{\RR ^d}\,\mathbf{1}_{E}(x)\,\mathbf{1}_E (x-y) \ud x,\quad y\in \RR^d.
\end{align}
It is a symmetric and uniformly continuous function tending to zero at infinity. If $E$ is of finite perimeter then it is Lipschitz continuous. For more details we refer to \cite{Galerne}. According to \cite[Lemma 1]{Heat_Content_1}, if $E$ is of finite Lebesgue measure and of finite perimeter then $\Per_\nu (E)<\infty$ for any measure $\nu$ satisfying \eqref{ass_1}.

We first aim to prove the following result. 
\begin{theorem}\label{Thm:Per-Approx_Lebesgue}
Let $\{\nu_\varepsilon\}_{\varepsilon >0}$ be a family of measures satisfying \eqref{ass_1} and let 
\begin{align*}
\lambda_\varepsilon (\mathrm{d}x) = C_\varepsilon^{-1}\left(1 \wedge R_\varepsilon |x|\right)\nu_\varepsilon (\mathrm{d}x)
\end{align*}
where $C_\varepsilon = \int_{\mathbb{R}^d}\left( 1 \wedge R_\varepsilon |x|\right)\nu_\varepsilon (\mathrm{d}x)$ and $R_\varepsilon \in [1,\infty]$.
We assume that, for any $R>0$,
\begin{align}\label{Lambda_cond_infty}
\lim_{\varepsilon \downarrow 0}
\lambda_\varepsilon (B_R^c)
=
1.
\end{align}
Then for any set $E\subset \RR^d$ of finite Lebesgue measure and of finite perimeter it holds
\begin{align}\label{per_to_Leb}
\lim_{\varepsilon \downarrow 0} C_{\varepsilon}^{-1}\mathrm{Per}_{\nu_{\varepsilon}}(E) = |E|.
\end{align}
Furthermore, if $\{\nu_\varepsilon\}_{\varepsilon >0}$ is a family of  finite measures then for any set $E$ of finite Lebesgue measure it holds
\begin{align*}
\lim_{\varepsilon \downarrow 0}
(\nu_\varepsilon (\RR^d))^{-1}
\mathrm{Per}_{\nu_{\varepsilon}}(E) = |E|.
\end{align*}
\end{theorem}

\begin{proof}
We have 
\begin{align*}
\Per_{\nu_{\varepsilon}}(E) 
&=
\int_{\RR^d}|E\cap (E^c-y)|\nu_{\varepsilon}(\mathrm{d}y)
=
\int_{\RR^d} (g_E(0)-g_E(y))
\nu_{\varepsilon}(\mathrm{d}y) \\
&= |E|  \nu_{\varepsilon}(B_{R}^c) 
- \int_{B_R^c}g_E(y) \nu_{\varepsilon}(\mathrm{d}y)
+
\int_{B_R} (g_E(0)-g_E(y))
\nu_{\varepsilon}(\mathrm{d}y) .
\end{align*}
Since $R_\varepsilon \geq 1$, for any  $R>1$ we have $C_{\varepsilon}^{-1}\nu_{\varepsilon}(B_{R}^c) = \lambda_{\varepsilon}(B_R^c)$ which tends to one as $\varepsilon$ goes to zero.  
If $|E|<\infty$ then $g_E\in C_0(\RR^d)$\footnote{By $C_0(\RR^d)$ we denote the space of all continuous functions that vanish at infinity.}. Thus we can choose $R$ big enough so that $g_E(y)$ is smaller than any given $\epsilon>0$ for $|y|>R$. We are left with the last integral in the formula above. We have 
\begin{align}\label{g_bound}
C_{\varepsilon}^{-1} \int_{B_R} (g_E(0)-g_E(y))
\nu_{\varepsilon}(\mathrm{d}y)
&\leq 
C C_{\varepsilon}^{-1}\int_{B_R}(1\wedge |y|)\nu_{\varepsilon}(\mathrm{d}y)
\leq C\lambda_{\varepsilon}(B_R)\xrightarrow[\varepsilon \downarrow 0]{} 0,
\end{align}
where we used the fact that $g_E$ is Lipschitz continuous if $\Per(E)<\infty$. If $\nu_\varepsilon$ is a finite measure then we simply choose $R_\varepsilon =\infty$ and repeat the same reasoning as above. In \eqref{g_bound} we use the fact that $g_E$ is bounded.
\end{proof}

We present an analogous result for stable L\'{e}vy measures.

\begin{proposition}\label{Prop:stable}
Let $\nu_\alpha$ be the $\alpha$-stable L\'{e}vy measure given by
\begin{align}\label{alpha_stable_Levy_meas}
\nu_\alpha (A) =\alpha \int_{\mathbb{S}^{d-1}}\int_{0}^\infty \mathbf{1}_A(\varrho \theta) \varrho^{-1-\alpha}\mathrm{d}\varrho \, \eta (\mathrm{d}\theta),\quad \alpha \in (0,1),
\end{align}
where $\eta $ is a probability measure on $\mathbb{S}^{d-1}$.
For any set $E\subset \RR^d$ such that $|E|<\infty$ and $\Per_{\nu_{\alpha_0}}(E)<\infty$, for some $\alpha_0\in (0,1) $ it holds
\begin{align}\label{conv_stable_prop}
\lim_{\alpha \downarrow 0}\Per_{\nu_{\alpha}}(E) = |E|.
\end{align}
\end{proposition}

\begin{proof}
Representation \eqref{alpha_stable_Levy_meas} yields $\nu_\alpha (B_R^c)=R^{-\alpha}$, for any $R>0$. As in the proof of Theorem \ref{Thm:Per-Approx_Lebesgue} we have
\begin{align}\label{Per_alpha_decom}
\Per_{\nu_{\alpha}}(E)  &= |E|R^{-\alpha} - \int_{B_R^c}g_E(y) \nu_{\alpha}(\mathrm{d}y)
+
\int_{B_R} (g_E(0)-g_E(y))
\nu_{\alpha}(\mathrm{d}y) 
\end{align}
and the first integral is negligible as $g_E\in C_0(\RR^d)$. To estimate the second integral we observe that
\begin{align*}
\alpha^{-1} \int_{B_R} (g_E(0)-g_E(y))
\nu_{\alpha}(\mathrm{d}y) 
&=
\int_{\mathbb{S}^{d-1}}\int_0^R (g_E(0)-g_E(\varrho \theta))\varrho^{-1-\alpha} \mathrm{d}\varrho \, \eta (\mathrm{d}\theta).
\end{align*}
Since for $ 0<\varrho \leq R$ and $0\leq \alpha <\alpha_0$,
\begin{align}\label{rho_ineq}
\varrho ^{-1-\alpha}\leq R^{\alpha_0}\varrho^{-1-\alpha_0} ,
\end{align}
we can apply the dominated convergence theorem in the last equation and this implies
\begin{align*}
 \lim_{\alpha \to 0} 
 \alpha^{-1} \int_{B_R} (g_E(0)-g_E(y))
\nu_{\alpha}(\mathrm{d}y) 
= \int_{\mathbb{S}^{d-1}}\int_0^R (g_E(0)-g_E(\varrho \theta))\varrho^{-1} \mathrm{d}\varrho \, \eta (\mathrm{d}\theta).
\end{align*}
The last expression is finite in view of \eqref{rho_ineq} used for $\alpha=0$ and combined with the assumption that the perimeter $\Per_{\nu_{\alpha_0}}(E)$ is finite. 
\end{proof}

\begin{remark}
We could prove convergence in \eqref{conv_stable_prop} also in the case when the spherical part $\eta$ of the measure $\nu_\alpha$ given in \eqref{alpha_stable_Levy_meas} depends on the parameter $\alpha$ and converges weakly towards some measure on the unit sphere, see e.g.\ \eqref{stable_Ludiwg}.
\end{remark}
Let the measure $\nu_\alpha$ be rotationally invariant and given by \eqref{rot_inv_stable_lev_meas}.
The following result provides an enhancement of \cite[Corollary 2.6]{Figalli} (see also \cite{Mazya}) for the classical $\alpha$-perimeter in the sense that we abandon the assumption of boundedness.

\begin{corollary}[Convergence of $\alpha$-perimeters as $\alpha \downarrow 0$]
Let 
$E\subset \RR^d$ be of finite Lebesgue measure and such that 
$\Per_{\nu_{\alpha_0}}(E)<\infty$, for some $\alpha_0\in (0,1) $. Then
\begin{align*}
\lim_{\alpha \downarrow 0}\alpha\Per_{\alpha} = \kappa_{d-1}|E|.
\end{align*}
\end{corollary}
\begin{proof}
This follows directly from Proposition \ref{Prop:stable} and equation \eqref{Our_Per_to_Alpha}.
\end{proof}

In the following example we show that the assumption of finite perimeter in Theorem \ref{Thm:Per-Approx_Lebesgue} cannot be weakened.

\begin{example}

We first present an example of a set $E\subset \RR$ of finite Lebesgue measure and such that its classical perimeter is infinite whereas  its $\alpha$-perimeter is finite for each $\alpha\in (0,1)$. We consider the one-dimensional case for simplicity's sake.
 Let
\begin{align}\label{set_E}
E=\bigcup_{n=1}^\infty \, [n,n+2^{-n}].
\end{align}
Clearly $|E|=1$ and $\Per (E)=\infty$. In order to compute $\Per_{\nu_\alpha}(E)$ we apply formula \eqref{Per_alpha_decom} which leads to
\begin{align}\label{local_Per_alpha_formula}
\Per_{\nu_{\alpha}}(E)  &= |E|- \int_{|y|>1}g_E(y) \nu_{\alpha}(\mathrm{d}y)
+
\int_{|y|\leq 1} (g_E(0)-g_E(y))
\nu_{\alpha}(\mathrm{d}y).
\end{align}
The first integral is clearly finite so it suffices to handle the last integral. We easily show that for $n\in \mathbb{N}$
\begin{align*}
 (g_E(0)-g_E(y))= (n-1)y +2^{-n+1},\quad \mathrm{for}\ y\in (2^{-n},2^{-n+1}).
\end{align*}
This implies
\begin{align*}
\int_{0}^1 (g_E(0)-g_E(y))
\nu_{\alpha}(\mathrm{d}y) &= 
\alpha\sum_{n=1}^\infty \left((n-1)\int_{2^{-n}}^{2^{-n+1}}\!\!\! y^{-\alpha}\mathrm{d}y
+
2^{-n+1}\int_{2^{-n}}^{2^{-n+1}} \!\!\!\! y^{-\alpha-1}\mathrm{d}y\right)\\
&=
\frac{\alpha}{1-\alpha}\frac{1}{1-2^{\alpha -1}}+ \frac{2^\alpha (1-2^{-\alpha})}{1-2^{\alpha -1}}\\
&\sim \frac{1}{(1-\alpha)^2} + \frac{1}{1-\alpha},\quad \alpha \uparrow 1.
\end{align*}
In particular, we infer that $\Per_{\nu_\alpha}(E)<\infty$ and $\lim_{\alpha \uparrow 1}\Per_{\nu_\alpha}(E)=\infty$.

Further, we consider a family of stable L\'{e}vy measures given by
\begin{align*}
\nu_n(\mathrm{d}x) = \nu_{\frac{1}{n}}(\mathrm{d}x)+ c_n\nu_{1-\frac{1}{n}}(\mathrm{d}x),\quad n\in \mathbb{N},
\end{align*}
where $\nu_{\frac{1}{n}}$ (resp.\ $\nu_{1-\frac{1}{n}}$) is the $\frac{1}{n}$-stable (resp.\ $(1-\frac{1}{n})$-stable) L\'{e}vy measure defined at \eqref{alpha_stable_Levy_meas} and $c_n$ is a sequence of positive numbers to be specified later. Using \eqref{local_Per_alpha_formula} and the above calculation we obtain for set $E$ given in \eqref{set_E} 
\begin{align*}
\Per_{\nu_n}(E)\sim \frac{1}{(1-\frac{1}{n})^2}+c_nn^{2}\sim c_nn^2,\quad n\to \infty.
\end{align*}
If $c_n\sim n^{-1}$ then $\lim_{n\to \infty}\Per_{\nu_n}(E)=\infty$ whereas for $c_n\sim n^{-2}$, $\lim_{n\to \infty}\Per_{\nu_n}(E)<\infty$.
Moreover, if we choose $c_n=o(n)$ then the corresponding sequence of measures 
\begin{align*}
\lambda_n (\mathrm{d}x) = C_n^{-1}(1\wedge |x|)\nu_n(\mathrm{d}x),
\end{align*}
with $C_n=\int (1\wedge |x|)\nu_n(\mathrm{d}x)$, satisfies condition $\lambda_n\xrightarrow[]{w}\delta_\infty$ of Theorem \ref{Thm:Per-Approx_Lebesgue}. Indeed, we easily find that
\begin{align*}
\int_{|x|>R}\nu_n(\mathrm{d}x) =
\begin{cases}
R^{-1/n}+c_nR^{1/n-1},\ R>1;\\
\frac{1/n}{1-\frac{1}{n}}(1-R^{-1/n+1})+ \frac{1-\frac{1}{n}}{1/n}(1-R^{1/n})+c_n+1,\ R<1
\end{cases}
\end{align*}
and
\begin{align*}
C_n = \frac{1/n}{1-\frac{1}{n}}+c_n\frac{1-\frac{1}{n}}{1/n}+c_n+1\to 1,\quad n\to \infty.
\end{align*}
It follows that under condition $c_nn\to 0$,
\begin{align*}
\lim_{n\to \infty} \lambda_n(B_R^c)=1,\quad R>0,
\end{align*}
which yields $\lim_{n\to \infty}\Per_{\nu_n}(E)=1$. 
\end{example}

As a further application of Theorem \ref{Thm:Per-Approx_Lebesgue} we obtain asymptotics for the perimeter given through rescaled measures under the assumption that the tail of the original measure is slowly varying at zero.

\begin{proposition}\label{prop:rescale_Leb}
Let $\nu$ be a given measure and let $\nu_\varepsilon (A) = \nu(\varepsilon A)$ for any $\varepsilon >0$. Assume that $\nu_\varepsilon$ satisfies \eqref{ass_1} for all $\varepsilon >0$. 
Suppose that the function $\ell(s) = \int_{|x|>s}\nu (\ud x)$ is slowly varying at zero. Then, for any set $E\subset \RR^d$ of finite Lebesgue measure and of finite perimeter it holds
\begin{align*}
\lim_{\varepsilon \downarrow 0} C_{\varepsilon}^{-1}\mathrm{Per}_{\nu_{\varepsilon}}(E) = |E|,
\end{align*}
where $C_\varepsilon = \int_{\RR^d}(1\wedge |x|)\nu_{\varepsilon}(\ud x)$.
\end{proposition}

\begin{proof}
We apply Theorem \ref{Thm:Per-Approx_Lebesgue} with $R_\varepsilon =1$. It is enough to show condition \eqref{Lambda_cond_infty}. For any $R\geq 1$ we have
\begin{align*}
\lambda_\varepsilon (B_R)
 &= \frac{\int_{|x|<R}(1\wedge |x|)\, \nu_\varepsilon (\ud x)}{\int_{\RR^d}(1\wedge |x|)\, \nu_\varepsilon (\ud x)} 
=
\frac{\int_{|x|<1} |x|\nu_\varepsilon (\ud x)+ \int_{1<|x|<R}\nu_\varepsilon (\ud x)}{\int_{|x|<1} |x|\nu_\varepsilon (\ud x)+ \int_{|x|>1} \nu_\varepsilon (\ud x)}.
\end{align*}
Further, 
\begin{align*}
\int_{|x|<1} |x|\nu_\varepsilon (\ud x) 
&= 
\int_{|x|<1} \int_0^{|x|}\ud u\, \nu_\varepsilon (\ud x) = \int_0^1 \int_{u<|x|<1}\nu_\varepsilon (\ud x)\, \ud u\\
&=
\int_0^1 \int_{\varepsilon u<|x|<\varepsilon}\nu (\ud x)\, \ud u
=
\int_0^1 (\ell (\varepsilon u)-\ell(\varepsilon))\, \ud u .
\end{align*}
This implies 
\begin{align}\label{slow_var_exp}
\lambda_\varepsilon (B_R)
 &=
 \frac{\int_0^1 \ell (\varepsilon u)\, \ud u -\ell (\varepsilon R)}{\int_0^1 \ell (\varepsilon u)\, \ud u}.
\end{align}
We set $L(w) = \ell (1/w)$, for $w>0$. By a change of variable we obtain 
$$
\int_0^1 \ell (\varepsilon u)\, \ud u = \frac{1}{\varepsilon}\int_{1/\varepsilon}^{\infty}w^{-2}L(w)\, \ud w.
$$
According to \cite[Proposition 1.5.10]{Bingham} we have
\begin{align*}
\frac{1}{\varepsilon}\int_{1/\varepsilon}^{\infty}w^{-2}L(w)\, \ud w
\sim L(1/\varepsilon),\quad \varepsilon \downarrow 0,
\end{align*}
which yields $\int_0^1 \ell (\varepsilon u)\, \ud u \sim \ell (\varepsilon)$ when $\varepsilon \downarrow 0$ and thus the  expression in \eqref{slow_var_exp} tends to zero as $\ell$ is slowly varying.
\end{proof}

We finish this paragraph with two results for $J$-perimeters.
\begin{corollary}\label{cor:slow_var}
Let $J\colon \RR^d \to (0,\infty)$ be a given kernel and let $J_\varepsilon (x)=\varepsilon^dJ(\varepsilon x)$, for $\varepsilon >0$. Assume that $(1\wedge |x|)J_\varepsilon (x)\in L^1(\RR^d)$ for all $\varepsilon >0$ and that $\ell (s) = \int_{|x|>s}J(x)\ud x$ is slowly varying at zero. Then, for any set $E\subset \RR^d$ of finite Lebesgue measure and of finite perimeter, it holds 
\begin{align*}
\lim_{\varepsilon \downarrow 0}\ell (\varepsilon )^{-1}\Per_{J\varepsilon}(E) = |E|.
\end{align*}
\end{corollary}
\begin{proof}
This follows from Proposition \ref{prop:rescale_Leb} if we choose $\nu_\varepsilon (\ud x) = J_\varepsilon (x)\ud x$ and notice that in this case $C_\varepsilon = \int (1\wedge |x|)J_\varepsilon (x)\, \ud x = \ell (\varepsilon )(1+o(1))$.
\end{proof}

The following easy observation is a consequence of Theorem \ref{Thm:Per-Approx_Lebesgue} if we choose $\nu_{\varepsilon}(\ud x) = J_{\varepsilon}(x)\ud x$ and observe that $\nu_{\varepsilon}(\RR^d)= \Vert J\Vert_{L^1}$. We remark, however, that it can be proved directly if we use the definition of the $J$-perimeter together with the dominated convergence theorem.
\begin{corollary}
Let $J\colon \RR^d \to (0,\infty)$ be a kernel such that $J\in L^1(\RR^d)$ and let $J_\varepsilon (x) = \varepsilon^d J(\varepsilon x)$, for $\varepsilon >0$. For any set $E\subset \RR^d$ of finite Lebesgue measure it holds 
\begin{align*}
\lim_{\varepsilon \downarrow 0}\Per_{J_\varepsilon}(E) = \Vert J\Vert_{L^1}|E|.
\end{align*}
\end{corollary}

\subsection*{Convergence of anisotropic fractional perimeters}
We finally present how to deduce convergence of anisotropic fractional $\alpha$-perimeters when $\alpha \downarrow 0$. We actually slightly improve on \cite[Theorem 6]{Ludwig_Per} as we do not require the set $E$ to be bounded. 
\begin{proposition}\label{result_Ludwig_at 0}
For any $E$ of finite Lebesgue measure and of finite perimeter it holds
\begin{align}\label{Ludwig at zero_new}
\lim_{\alpha \downarrow 0}\alpha \Per_{\alpha}(E,K) = d|K||E|.
\end{align}
\end{proposition}
\begin{proof}
We consider the following  measure
\begin{align}\label{stable_Ludiwg}
\nu_\alpha (A, K) =\alpha \int_{\mathbb{S}^{d-1}}\int_{0}^\infty \mathbf{1}_A(\varrho \theta) \varrho^{-1-\alpha}\mathrm{d}\varrho \, \frac{\mathrm{d}\theta}{\Vert \theta \Vert_K^{d+\alpha}},\quad \alpha \in (0,1).
\end{align}
We observe that
\begin{align}\label{Per-rel}
\Per _{\nu_{\alpha}(\cdot ,K)}(E) = \alpha \Per_\alpha (E,K).
\end{align}
Proceeding similarly as in \eqref{Per_alpha_decom} we find that
\begin{align*}
\Per_{\nu_\alpha (\cdot ,K)}(E) = C_\alpha (K)R^{-\alpha}|E| - \int_{B_R^c}g_E(y) \nu_{\alpha}(\mathrm{d}y, K)
+
\int_{B_R} (g_E(0)-g_E(y))
\nu_{\alpha}(\mathrm{d}y, K), 
\end{align*}
where $C_\alpha (K) = \int_{\mathbb{S}^{d-1}}\frac{\ud \theta}{\Vert \theta \Vert_K^{d+\alpha}}$. The integral over $B^c_R$ may be done arbitrarily small if we choose $R$ big enough.  
Since $\Per (E)<\infty$, we know that $g_E(x)$ is a  Lipschitz function. This and the fact that 
 $\Vert \theta \Vert ^{-d-\alpha }_K$ is bounded for $\theta \in \mathbb{S}^{d-1}$ allows us to apply the dominated convergence theorem and we obtain
\begin{align*}
\lim_{\alpha \downarrow 0}\alpha ^{-1}\int_{B_R} (g_E(0)-g_E(y))
\nu_{\alpha}(\mathrm{d}y, K) 
= 
\int_{\mathbb{S}^{d-1}}\int_0^R (g_E(0)-g_E(r\theta))r^{-1}\, \ud r\, \frac{\ud \theta}{\Vert \theta \Vert_K^{d}}.
\end{align*}
Further, we have
\begin{align*}
\lim_{\alpha \downarrow 0}C_\alpha (K) = \int_{\mathbb{S}^{d-1}}\frac{\ud \theta}{\Vert \theta \Vert_K^d}= d|K|
\end{align*}
and we infer that 
\begin{align*}
\lim_{\alpha \downarrow 0}\Per_{\nu_{\alpha}(\cdot ,K)}(E) = d|K||E|,
\end{align*}
which in view of \eqref{Per-rel} implies \eqref{Ludwig at zero_new}.
\end{proof}

\section{Appendix: Proof of Theorem \ref{Thm:Ponce}}\label{sec:Appendix}
We aim to establish a more general version of \cite[Theorem 2]{Ponce} which is included in Theorem \ref{Thm:Ponce}. We follow closely the approach of \cite{Ponce}.  In the remaining part of the text we make use of mollifiers. Let $j\in L^1(\RR^d)$ be such that $\int_{\RR^d}j(x)\ud x =1$. For any $\delta >0$ let $j_\delta (x) =\delta^{-d}j(x/\delta)$. Clearly, $\int_{\RR^d}j_\delta (x)\ud x=1$. For any $f\in L^1(\RR^d)$, we set
\begin{align*}
f_\delta (x) = j_{\delta }\ast f(x),\quad x\in \RR^d.
\end{align*}
We start with the following series of lemmas.

\begin{lemma}\label{Lemma_Ponce_1}
For any probability measure $\lambda$ on $\RR^d$ and any $f\in \mathrm{BV} (\RR^d)$ it holds
\begin{align*}
\int_{\RR^d}\int_{\RR^d} \frac{|f(x+y) - f(x)|}{|y|}\lambda (\ud y)\, \ud x \leq 
\int_{\RR^d}\int_{\RR^d} 
\Big\vert Df \cdot \frac{y}{|y|}\Big\vert \, \lambda (\ud y) .
\end{align*}
\end{lemma}

\begin{proof}
For any $R>0$ and $\delta >0$, by a standard argument basing upon the Fundamental Theorem of Calculus, we have
\begin{align*}
\int_{B_R}\int_{B_R} \frac{|f_\delta (x+y)-f_\delta (x)|}{|y|}&\lambda (\ud y)\, \ud x\\
&\leq 
\int_{B_R}\int_{B_R}  \int_0^1 \left\vert \nabla f_\delta (x+\rho y)\cdot \frac{y}{|y|}\right\vert \ud \rho \, \lambda (\ud y)\, \ud x \\
&\leq 
\int_{\RR^d}\int_0^1\int_{\RR^d} \left\vert \nabla f_\delta (x+\rho y)\cdot \frac{y}{|y|}\right\vert \ud x \, \ud \rho \, \lambda (\ud y)\\
&=
\int_{\RR^d} \int_{\RR^d} 
 \left\vert \nabla f_\delta (x)\cdot \frac{y}{|y|}\right\vert \ud x  \, \lambda (\ud y)\\
 &=
 \int_{\RR^d} \int_{\RR^d} \left\vert \int_{\RR^d} j_\delta (x-z)\,  Df(\ud z)\cdot \frac{y}{|y|}\right\vert  \ud x  \, \lambda (\ud y)\\
 &=
  \int_{\RR^d} \int_{\RR^d} \left\vert \int_{\RR^d} j_\delta (x-z) \frac{Df}{|Df|}(z)\cdot \frac{y}{|y|}\, |Df|(\ud z)\right\vert  \ud x  \, \lambda (\ud y)\\
  &\leq
   \int_{\RR^d} \int_{\RR^d} \int_{\RR^d}
   j_\delta (x-z) \left\vert \frac{Df}{|Df|}(z)\cdot \frac{y}{|y|}\right\vert \ud x\,  |Df|(\ud z) \, \lambda (\ud y)\\
   &=
    \int_{\RR^d} \int_{\RR^d} 
    \left\vert \frac{Df}{|Df|}(z)\cdot \frac{y}{|y|}\right\vert  |Df|(\ud z)\, \lambda (\ud y)\\
    &=
    \int_{\RR^d} \int_{\RR^d} 
    \left\vert Df\cdot \frac{y}{|y|}\right\vert \lambda (\ud y),
\end{align*}
where we used \eqref{Radon-Nikodym} together with \cite[Proposition 1.23]{Ambrosio_2000}. We finally take $\lambda \downarrow 0$ and then $R\to \infty $ and the result follows.
\end{proof}

\begin{lemma}\label{Lemma_Ponce_2}
For any probability measure $\lambda$ on $\RR^d$ and any $f\in L^1(\RR^d)$ it holds
\begin{align*}
\int_{\RR^d}\int_{\RR^d} \frac{|f_\delta (x+y) - f_\delta (x)|}{|y|}\lambda (\ud y)\, \ud x 
\leq
\int_{\RR^d}\int_{\RR^d} \frac{|f(x+y) - f(x)|}{|y|}\lambda (\ud y)\, \ud x .
\end{align*}
\end{lemma}

\begin{proof}
We clearly have
\begin{align*}
\int_{\RR^d}\int_{\RR^d} \frac{|f_\delta (x+y) - f_\delta (x)|}{|y|}&\lambda (\ud y)\, \ud x \\
&\leq
\int_{\RR^d}\int_{\RR^d} \int_{\RR^d}\frac{|f (x+y-z) - f (x-z)|}{|y|}\, j_\delta (z)\, \ud z\, \lambda (\ud y)\, \ud x \\
&=
\int_{\RR^d}\int_{\RR^d} \int_{\RR^d}\frac{|f (w+y) - f (w)|}{|y|}\, j_\delta (x+w)\,\ud x\,  \ud w\lambda (\ud y)\\
&=
\int_{\RR^d}\int_{\RR^d} \frac{|f(w+y) - f(w)|}{|y|}\,  \ud w\, \lambda (\ud y)
\end{align*}
and the proof is finished.
\end{proof}

Recall that $\{\lambda\}_{\varepsilon >0}$ is a family of probability measures on $\RR^d$ such that $\lambda_\varepsilon (\{0\})=0$ and, for any $R>0$,
\begin{align}\label{lambda_to_zero}
\lim_{\varepsilon \downarrow  0}\lambda_\varepsilon (B_R^c)=0.
\end{align}
We consider $\mu_{\varepsilon}(E)= \lambda_{\varepsilon}((0,\infty) E)$, where $(0,\infty) E = \{re:\, e\in E \ \text{and}\ r>0\}$ is a cone spanned by $E\subset \mathbb{S}^{d-1}$. Since the family $\{\mu_\varepsilon\}_{\varepsilon >0}$ is bounded in the space of Radon measures on $\mathbb{S}^{d-1}$, there exists a sequence $\varepsilon_j$ converging to zero and a probability measure $\mu $ on the unit sphere such that $\mu_{\varepsilon_j}\xrightarrow{w}\mu$. We observe that in view of the definition, for any continuous function $F\colon \mathbb{S}^{d-1}\to \RR$ we have
\begin{align}
\int_{\RR^d}F\Big( \frac{y}{|y|}\Big)\lambda_\varepsilon (\ud y) = \int_{\mathbb{S}^{d-1}}F(\theta)\mu_\varepsilon (\ud \theta),\quad \varepsilon >0.
\end{align}
It evidently follows that
\begin{align}\label{lambda_conv_sphere}
\lim_{j\to \infty}
\int_{\RR^d}F\Big( \frac{y}{|y|}\Big)\lambda_{\varepsilon_j} (\ud y) = \int_{\mathbb{S}^{d-1}}F(\theta)\mu (\ud \theta).
\end{align}

\begin{lemma}\label{Lemma_Ponce_3}
Let $B_R\subset \RR^d$ be an arbitrary Euclidean ball of radius $R>0$. Then, for any $f\in C^2(B_R)$, it holds
\begin{align*}
\lim_{j\to \infty}
\int_{B_R}\int_{B_R} \frac{|f(x+y) - f(x)|}{|y|}\lambda_{\varepsilon_j} (\ud y)\, \ud x
=
\int_{\mathbb{S}^{d-1}}\int_{B_R} |\nabla f(x) \cdot \theta |\, \ud x\, \mu(\ud \theta) . 
\end{align*} 
\end{lemma}

\begin{proof}
We start with the following easy inequality
\begin{align*}
\left\vert 
\frac{|f(x+y)-f(y)|}{|y|}- \big\vert \nabla f(x)\cdot \frac{y}{|y|}\big\vert
\right\vert
&\leq
\frac{\big\vert f(x+y)-f(x)-\nabla f(x)\cdot \frac{y}{|y|}\big\vert}{|y|}\\
&\leq C|y|,\quad C>0.
\end{align*}
This implies
\begin{multline*}
\int_{B_R} \int_{B_R} 
\left\vert 
\frac{|f(x+y)-f(y)|}{|y|}- \big\vert \nabla f(x)\cdot \frac{y}{|y|}\big\vert
\right\vert
\lambda_\varepsilon (\ud y)\, \ud x\\
\leq
|B_R|\left( C\int_{|y|\leq 1}|y|\lambda_\varepsilon (\ud y) +2\Vert \nabla f\Vert_\infty \int_{|y|>1}\lambda_\varepsilon (\ud y)\right).
\end{multline*}
By \eqref{lambda_to_zero}, the second integral converges to zero as $\varepsilon \downarrow 0$. For the second we write
\begin{align*}
\int_{|y|\leq 1}|y|\lambda_\varepsilon (\ud y) 
&=
\int_{|y|\leq 1} \int_0^{|y|} \ud t\, \lambda_\varepsilon (\ud y) = \int_0^1 \int_{t\leq |y|\leq 1}\lambda_\varepsilon (\ud y)\, \ud t\\
&\leq 
\int_0^1 \int_{ |y|>t}\lambda_\varepsilon (\ud y)\, \ud t\to 0,\quad \mathrm{as}\ \varepsilon \downarrow 0.
\end{align*}
Thus we are left to show that
\begin{align*}
\lim_{j\to \infty} 
\int_{B_R} \int_{B_R}
\big\vert \nabla f(x)\cdot \frac{y}{|y|}\big\vert
\lambda_{\varepsilon_j} (\ud y)\, \ud x
=
\int_{\mathbb{S}^{d-1}}\int_{B_R} |\nabla f(x) \cdot \theta |\, \ud x\, \mu(\ud \theta) . 
\end{align*}
We have
\begin{align*}
\int_{B_R} \int_{\RR^d}
\big\vert \nabla f(x)\cdot \frac{y}{|y|}\big\vert
\lambda_{\varepsilon_j} (\ud y)\, \ud x
&=
\int_{B_R} \int_{B_R}
\big\vert \nabla f(x)\cdot \frac{y}{|y|}\big\vert
\lambda_{\varepsilon_j} (\ud y)\, \ud x\\
&\quad \quad +
\int_{B_R} \int_{B_R^c}
\big\vert \nabla f(x)\cdot \frac{y}{|y|}\big\vert
\lambda_{\varepsilon_j} (\ud y)\, \ud x.
\end{align*}
As the last integral clearly tends to zero, for $\varepsilon \downarrow 0$, we infer the result with the aid of \eqref{lambda_conv_sphere} and the dominated convergence theorem.
\end{proof}

\begin{proof}[Proof of Theorem \ref{Thm:Ponce}]
By Lemma \ref{Lemma_Ponce_1} and Lemma \ref{Lemma_Ponce_2} we have
\begin{multline*}
\int_{B_{1/R}} \int_{B_{1/R}}
\frac{|f_\delta (x+y) - f_\delta (x)|}{|y|}\lambda_{\varepsilon_j} (\ud y)\, \ud x \\
\leq 
\int_{\RR^d} \int_{\RR^d} 
\frac{|f (x+y) - f (x)|}{|y|}\lambda_{\varepsilon_j} (\ud y)\, \ud x\\
\leq 
\int_{\RR^d}\int_{\RR^d} 
\Big\vert Df \cdot \frac{y}{|y|}\Big\vert \, \lambda_{\varepsilon_j} (\ud y).
\end{multline*}
We note that the following convergence holds uniformly for $\theta \in \mathbb{S}^{d-1}$
\begin{align*}
\lim_{\delta \downarrow 0}\int_{\RR^d}|\nabla f_\delta \cdot \theta| = \int_{\RR^d}|Df \cdot \theta|.
\end{align*}
Combining this with Lemma \ref{Lemma_Ponce_3} 
 we obtain
\begin{align*}
\lim_{R\downarrow 0}\lim_{\delta \downarrow 0}\lim_{j\to \infty}
\int_{B_{1/R}} \int_{B_{1/R}}&
\frac{|f_\delta (x+y) - f_\delta (x)|}{|y|}\lambda_{\varepsilon_j} (\ud y)\, \ud x 
=
\int_{\mathbb{S}^{d-1}}\int_{\RR^d}|Df\cdot \theta|\mu (\ud \theta).
\end{align*}
Since the function $\mathbb{S}^{d-1}\ni \theta \mapsto \int_{\RR^d}|Df \cdot \theta|$ is  continuous, we can apply \eqref{lambda_conv_sphere}  and we arrive at
\begin{align*}
\lim_{j\to \infty} 
\int_{\RR^d}\int_{\RR^d} 
\Big\vert Df \cdot \frac{y}{|y|}\Big\vert \, \lambda_{\varepsilon_j} (\ud y)
=
\int_{\mathbb{S}^{d-1}}\int_{\RR^d}|Df\cdot \theta|\mu (\ud \theta),
\end{align*}
and the proof is finished.
\end{proof}

\section*{Acknowledgement}
We wish to thank R.\ L.\ Schilling (TU Dresden) for stimulating discussions and helpful comments.

\bibliographystyle{abbrv}
\bibliography{non-local_per_Final}

\end{document}